\documentclass[12pt]{amsart}
\usepackage{amssymb,amsthm,amsmath,amstext}
\usepackage{mathrsfs}  % allows nice script font for sheaves
\usepackage{bm}        % allows bold italic letters in math mode
\usepackage{mathtools} % allows more extendible arrows

\usepackage[left=1.28in,top=.8in,right=1.28in,bottom=.8in]{geometry}

\usepackage{graphicx}
\usepackage[all]{xy}
\theoremstyle{plain}
\newtheorem{theorem}{Theorem}[section]
\newtheorem{prop}[theorem]{Proposition}
\newtheorem{lemma}[theorem]{Lemma}
\newtheorem{question}[theorem]{Question}
\newtheorem{cor}[theorem]{Corollary}
\newcommand{\lBr}{{}_\ell\mathrm{Br}}

\newtheorem{theoremintro}{Theorem}

\theoremstyle{definition}

\theoremstyle{remark}
\newtheorem{remark}[theorem]{Remark}

\newcommand{\sheaf}[1]{\mathscr{#1}}

\newcommand{\PP}{\sheaf{P}}
\newcommand{\XX}{\sheaf{X}}
\newcommand{\YY}{\sheaf{Y}}

\newcommand{\Z}{\mathbb Z}

\DeclareMathOperator{\Br}{\mathrm{Br}}

\newcommand{\et}{\mathrm{\acute{e}t}}

\usepackage{hyperref}
 
\hypersetup{colorlinks=true,linkcolor=blue,anchorcolor=blue,citecolor=blue,letterpaper=true}

\begin{document}

\title[$u$-invariant and uniform bound ]
{ On the $u$-invariant  of  function fields of curves   over complete discretely 
valued fields}

\author[Parimala]{R.\ Parimala }
\address{Department of Mathematics \& Computer Science \\ %
Emory University \\ %
400 Dowman Drive~NE \\ %
Atlanta, GA 30322, USA}
\email{ parimala@mathcs.emory.edu, suresh@mathcs.emory.edu}
 
\author[Suresh]{V.\ Suresh}

\date{}

\begin{abstract}  Let $K$ be a complete discretely  valued field with residue field
$\kappa$. If char$(K) = 0$, char$(\kappa) = 2$ and $[\kappa : \kappa^2] = d$, we prove that 
there exists an integer $N$ depending on $d$ such that  the $u$-invariant of 
 any function field in one variable over $K$  is bounded by $N$.  The method of proof is via
 introducing  the notion of uniform boundedness for the $p$-torsion of the Brauer group of a field and relating  the
 uniform boundedness of the 2-torsion of the Brauer group to finiteness of the $u$-invariant. 
 We prove  that the 2-torsion of the Brauer group of function fields  in one variable over $K$ are uniformly
 bounded.  
\end{abstract} 
 
\maketitle

\def\ZZ{${\mathbf Z}$}
\def\ih{${\mathbf H}$}
\def\RR{${\mathbf R}$}
\def\IF{${\mathbf F}$}
\def\IP{${\mathbf P}$}

Let $K$ be a complete discretely  valued field with residue field
$\kappa$ and $F$ a function field in one variable over $K$. Suppose
char$(\kappa) \neq 2$. A bound for the $u$-invariant of $F$ in terms
of the $u$-invariant of function fields in one variable over $\kappa$
were obtained by Harbater-Hartmann-Krashen \cite{hhk:patching_qf_csa} using patching
techniques. This recovers the $u$-invariant of  function fields of
non-dyadic  $p$-adic curves (\cite{ps:uinv_8}).  Leep
(\cite{leep}), using results of Heath-Brown (\cite{heath-brown}),
proved that  the $u$-invariant of function
fields of all $p$-adic curves (including dyadic curves) is 8.  An alternate  proof for
function fields of dyadic curves is given in (\cite{ps:period_index}). In fact  more
generally we proved that if char$(K) = 0$, char$(\kappa) =2$ and $\kappa$ is perfect,
then $u(F) \leq 8$.  If $[\kappa : \kappa^2]$ is infinite it is easy to
construct anisotropic quadratic forms  over $K$ and hence over $F$ of arbitrarily large
dimension. The question remained open whether the $u$-invariant of $F$
is finite if char$(\kappa) = 2$ and  $[\kappa : \kappa^2]$ is
finite.  The aim of this article is to give an affirmative answer to
this question.  More precisely we prove the following (\ref{theorem:u_invariant})

\begin{theoremintro}
Let $K$ be complete discretely  valued field with residue field
$\kappa$ and $F$ a function field of a curve over $K$. 
Suppose that char$(K)  = 0$, char$(\kappa) = 2$ and $[\kappa : \kappa^2]$ is finite.  
Then there exists an integer $M$ which depends only on $[\kappa :
\kappa^2]$ such that  for any finite extension $F$ of $K(t)$, $u(F) \leq M$.  
\end{theoremintro}

 It was conjectured in (\cite{ps:period_index}) that $u(F)$ is at most 
$8[\kappa: \kappa^2]$. 
The bound we give for the $u(F)$ is not effective and far from the conjectural bound.

Let $L$ be a field of characteristic not equal to 2 with $H^M(L,
\mu_2^{\otimes M}) = 0$ for some $M \geq 1$.  Suppose that there
exists an integer $N$  such that for all finite extensions $E$ of $L$
and for any $\alpha \in H^n(E, \mu_2^{\otimes n}), n \geq 2$, there
exists an extension $E'$  of $E$ of degree at most
$N$ such that $\alpha  \otimes_E E' = 0$. Then  a theorem of Krashen
(\ref{cor:krashen2})  asserts that the $u$-invariant
of $L$ is finite.  Our aim is to prove that if $K$ is a complete
discretely  valued field with residue field $\kappa$ of characteristic 2
and $[\kappa : \kappa^2]$ finite and $F$ a function field in one
variable  over $K$, then such an integer $N$ exists for $F$, thereby
proving the finiteness of the $u$-invariant of $F$.

We introduce the
notion of  uniform boundedness for the $\ell$-torsion of the Brauer
group $\Br(L)$ of $L$, where $L$ is any field. 
 We say that the Brauer group of  $L$ is {\it  uniformly
$\ell$-bounded}  if there exists an integer $N$ such that for any
finite extension $E$ of $L$ and for  any set of  finitely many elements  
$\alpha_1, \cdots , \alpha_n \in
\lBr(E)$, there is a finite  extension $E'$ of $E$ of degree at most $N$ such
that $\alpha_i \otimes_E E' = 0$ for $1 \leq i \leq n$.  Using the result of
Krashen (\ref{cor:krashen2}), we show that if $L$ is a field of
characteristic not equal to 2 with  $H^M(L, \mu_2^{\otimes M}) = 0$ for
some $M \geq 1$ and with  the Brauer group of $L$ is $2$-uniformly bounded,
then $u(L)$ is finite (\ref{cor:uniformbound_uinvariant}).
It looks plausible that there are fields $L$ of  finite $u$-invariant
with $L$ not  uniformly 2-bounded. 

The main result of the paper is to prove the uniform $p$-boundedness
for   the Brauer group of  any function field $F$ in
one variable over a complete discretely  valued field $K$ with residue field
$\kappa$, where char$(\kappa) = p$, char$(F) \neq p$ and $[\kappa :
\kappa^p]$ is finite.   We also prove the uniform $\ell$-boundedness
for the Brauer group of  any function field in
one variable over a complete discretely  valued field with residue field
$\kappa$ and $\ell \neq $char$(\kappa)$ under the assumption that the
Brauer groups of  $\kappa$ and $\kappa(t)$ are  uniformly
$\ell$-bounded. This  result  for function fields
of   $p$-adic  curves ($p \neq \ell$) is due to Saltman  (\cite{saltman:jrms}). 
 To prove our theorems  we  use  the patching
techniques of Harbater-Hartmann-Krashen and results of
(\cite{ps:period_index}).

We thank  D. Harbater for his very useful comments on the text.

\section{ Galois cohomology, Symbol length, $u$-invariant and Uniform
  bound}  

In this section we recall the recent results of Krashen and Saltman 
connecting the symbol length and effective index in Galois cohomology
with the  $u$-invariant of a field.

Let $K$ be a field and $\ell$ a prime not equal to the characteristic
of $K$.  Let $\mu_\ell$ denote the Galois module of $\ell^{\rm th}$
roots of unity and $H^n(K, \mu_\ell^{\otimes m})$ denote the $n^{\rm
  th}$ Galois cohomology group with values in $\mu_\ell^{\otimes m}$. 
We have $H^1(K, \mu_\ell)  \simeq K^*/K^{*\ell}$. For $a \in K^*$, let 
$(a) \in H^1(K, \mu_\ell)$ denote the  image of $aK^{*\ell}$. Let
$a_1, \cdots , a_n \in K^*$. The cup product  
$(a_1) \cdot (a_2) \cdots (a_n) \in H^n(K, \mu_\ell^{\otimes n})$ 
is called a {\it symbol} in $H^n(K, \mu_\ell^{\otimes n})$. A theorem of
 Voevodsky (\cite{voevodsky}) asserts that every element in $H^n(K,
\mu_\ell^{\otimes n})$ is a sum of symbols. 

Let $\alpha \in H^n(K,
\mu_\ell^{\otimes n})$. The {\it symbol length} of $\alpha$, denoted by
$\lambda(\alpha)$, is defined as the smallest $m$ such that $\alpha $
is a sum of $m$ symbols in $H^n(K, \mu_\ell^{\otimes n})$. For any
$\alpha \in H^n(K, \mu_\ell^{\otimes m})$, the {\it
  effective index} of $\alpha$, denoted by $eind(\alpha)$, is defined
to be the minimum of the degrees of finite  field extensions $E$ of $K$
with $\alpha_E = 0$, where $\alpha_E$ is the image of $\alpha$ in
$H^n(E, \mu_\ell^{\otimes m})$.  Since $H^2(K, \mu_\ell) \simeq
\lBr(K)$, for $\alpha \in \lBr(K)$, eind$(\alpha)$ is equal to the index of a central simple
algebra $A$ over $K$ representing $\alpha$. 
The following lemma asserts that this definition of
effective index coincides with the  definition in (\cite{krashen:symbol_length}). 

\begin{lemma}
\label{lemma:sep_ext} Let $K$ be a field  and $\ell$ a prime not equal to
  char$(K)$. Let $\alpha \in H^n(K, \mu_\ell^{\otimes
    m})$. Suppose that exists an extension $L$ of $K$ of degree at
  most $N$ with $\alpha \otimes_K L = 0$. Then there exists a separable
  field extension  $E$ of $K$ of degree at most $N$ such that $\alpha
  \otimes_K E = 0$.
\end{lemma}

\begin{proof} Let $L$ be an extension of $K$ of degree at most
  $N$ with $\alpha \otimes_KL  = 0$. Let  $E$ be the separable closure of
  $K$ in $L$. Let $p $ be the characteristic of $(K)$. Suppose $p > 0$. Then $L/E$ is of
  degree $p^r$ for some   $r \geq 0$.  Since   $\ell \neq
  p$, the restriction map $H^n(E, \mu_\ell^{\otimes m}) \to H^n(L,
  \mu_\ell^{\otimes m})$ is injective (\cite[Cor. on p.12]{serre:gc}). 
Hence     $\alpha \otimes_K E = 0$.
\end{proof}

We know recall a theorem of  Krashen
(\cite[4.2]{krashen:symbol_length}, cf. \cite{saltman:symbol_length}).

\begin{theorem}
\label{theorem:krashen1}
Let $K$ be a field and $\ell$ a prime not
  equal to  the characteristic of $K$.   Let $n \geq 1$.
Suppose that  there exists an  integer $N$
such that for every finite  extension $L$ of $K$    and for every element
$\beta \in H^d(L, \mu_\ell^{\otimes d})$, $ 1 
  \leq d \leq n-1$,  eind$(\beta) \leq N$. Then  for any  $\alpha \in H^n(K,
  \mu_\ell^{\otimes n})$,   $\lambda(\alpha)$ is bounded in term of
  eind$(\alpha)$, $N$ and $n$.  
 \end{theorem}

The following is  a consequence of the above theorem.

\begin{cor} 
\label{cor:krashen1}
Let $K$ be a field and $\ell$ a prime not
  equal to  the characteristic of $K$.  Let $n\geq 1$. Suppose that
  there exists an integer $N$ such that    for all finite extensions
  $L$ of $K$ and for all $\alpha \in H^d(L, \mu_\ell^{\otimes d})$, $1\leq d \leq
  n$, eind$(\alpha) \leq N$.  Then  there exists an integer $N'$ which
  depends only $N$ and $n$ such that $\lambda(\alpha) \leq N'$ for all
  finite extensions $L$ of $K$ and $\alpha
  \in H^n(K, \mu_\ell^{\otimes n})$. 
\end{cor}

Let $K$ be a field of characteristic not equal to 2. The {\it $u$-invariant}
of $K$ is defined to be  the supremum of dimensions of anisotropic
quadratic forms over $K$.  The following theorem is a  
consequence of a theorem of Orlov,  Vishik, and  Voevodsky
(\cite{ovv}) on the Milnor conjecture  (cf.  \cite{kahn},
\cite{ps:symbol_length}).  

\begin{theorem} 
\label{theorem:symbol_length_u_invariant}
Let $K$ be a field of characteristic not equal to 2.
Suppose that there exist integer $M\geq 1 $ and $N$ such that $H^M(K,
\mu_2) =0$ and $\lambda(\alpha) \leq N$  for all $\alpha \in H^d(K,
\mu_2^{\otimes d})$, $ 1 \leq d < M$. Then  the $u$-invariant is
bounded by a function of $M$ and $N$. 
\end{theorem}
 
The   following follows  from (\ref{theorem:symbol_length_u_invariant}) and
(\ref{cor:krashen1})  (cf. \cite[5.5]{krashen:symbol_length}).

\begin{cor} 
\label{cor:krashen2} 
 Let $K$ be a field of characteristic not equal to 2. 
Suppose that  there exist integers $N \geq 1$ and $M \geq
1$ such that for all finite  extensions $L$ of $K$, $H^{M}(L,
\mu_2) = 0$ and for all $ n \geq 1$ and  $\alpha \in H^n(L, \mu_2)$,
eind$(\alpha) \leq N$.  Then there exists an integer $N'$, which
depends only on $N$ and $M$, such that $u(K)  \leq N'$. 
\end{cor}

Let $K$ be a field  and $\ell$ a prime not equal to the characteristic
of $K$. We say that $K$ is $(n , \ell)$-{\it uniformly 
  bounded }  if there exists an integer $N$ such that 
for any finite extension $L$ of $K$ and  $\alpha_1, \cdots , \alpha_m
\in H^n(L, \mu_\ell^{\otimes   n})$ there is an extension  $E$ of $L$ with $[E : L] \leq N$ and
$\alpha_{i} \otimes_L E = 0$ for $1 \leq i \leq m$.  Such an $N$ is called an
$(n, \ell)$-{\it uniform bound} for $K$.   We note that if $N$ is an $(n, \ell)$-uniform 
boud for $K$, then $N$ is also an $(N, \ell)$-uniform bound for  any 
finite extension $L$ of $K$. 

In view of a theorem of
Voevodsky (\cite{voevodsky}) on the Bloch-Kato  conjecture,
every element in $H^n(K, \mu_\ell^{\otimes n})$ is a sum of
symbols. In particular $N$ is an $(n, \ell)$-uniform bound of 
$K$    if and only if    
for given symbols  $\alpha_1, \cdots , \alpha_m \in H^n(K, \mu_\ell^{\otimes
  n})$ there is an extension  $L$ of $K$ with $[L : K] \leq N$ and
$\alpha_{i_L} = 0$ for $1 \leq i \leq m$. 

\begin{lemma}
\label{lemma:uniform_bound_d_d+1} Let $K$ be a field  and $\ell$ a
prime not equal to the characteristic  of $K$. If $N$ is an 
$(n,\ell)$-uniform bound for $K $, then $N$ is also a $(d,
\ell)$-uniform bound for $K $ for all $d \geq n$. 
\end{lemma}

\begin{proof}  Suppose $N$ is an $(n, \ell)$-uniform bound for $ K$. 
It is enough to prove the lemma for $d = n+1$.  Let $L$ be a finite
extension of $K$ and 
$\alpha_1, \cdots , \alpha_m
\in H^{n+1} (L, \mu_\ell^{\otimes (n+1)})$ be  symbols.
 Then $\alpha_i =
\beta_i \cdot (a_i)$ for some symbols $\beta_i \in  H^{n}(L,
\mu_{\ell}^{\otimes n})$ and $a_i \in L^*$, $1 \leq i \leq m$.  
Since $N$ is an $(n, \ell)$-uniform bound for $K$,  
 there exists  a field extension  $E$ of
  $L$  with $[E : L] \leq N$ and $\beta_ {i_E} = 0$ for $1 \leq i \leq
  m$. Then clearly $\alpha_{i_E} = 0$ for $1 \leq i \leq m$. 
Thus  $N$ is also an $(n+1, \ell)$-uniform bound for
$K$.
\end{proof}

\begin{cor}
\label{cor:symbol_length} Let $K$ be a field and $\ell$ a prime not
equal to the characteristic of  $K$. Suppose that  $N$   is  a $(2, \ell)$-uniform
bound for $K$. Then for every $n \geq 2$, there exists an integer
$N_n$, which depends only on $N$ and $n$ such
that $\lambda(\alpha) \leq N_n$ for all $\alpha \in H^n(K,
\mu_\ell^{\otimes n})$. 
\end{cor}

\begin{proof}  Since $N$ is $(2, \ell)$-uniform  bound for $K$,    by
 ( \ref{lemma:uniform_bound_d_d+1}), $N$ is  an  $(n, \ell)$-uniform
 bound for  $K$  for all $n \geq 2$.   Let $\alpha \in H^n(K, \mu_\ell^{\otimes
   n})$. Then, by (\ref{theorem:krashen1}), $\lambda(\alpha)$ is bounded in
terms of eind$(\alpha)$, $N$ and $n$. Since eind$(\alpha) \leq N$,  $\lambda(\alpha)$
is bounded in terms of $N$ and $n$.  
\end{proof}

\begin{cor}
\label{cor:uniformbound_uinvariant} 
Let $K$ be a field  of characteristic not equal to
  2.  Suppose that there exists an integer $M$ such that for all
  finite extensions $L$ of $K$, $H^M(K,
  \mu_2)  = 0$ and $N$ is a $(2,2)$-uniform  bound for $K$. Then
  there exists  $N'$ which depends only on $N$ and $M$ such that for
  any finite extension $L$ of $K$,  $u(L) \leq N'$. 
\end{cor} 

\begin{proof} Since the conditions on $K$ are also satisfied by any
  finite extensions of $K$, it is enough to find an $N'$ which depends
  only on $N$ and $M$ such that $u(K) \leq N'$. 

Since $N$ is a $(2,2)$-uniform bound for $K$, by
  (\ref{cor:symbol_length}) there exist integers  $N_n$ for $1 \leq n
  < M$, which depends only on $N$ and $n$ such that
  for all $\alpha \in H^n(K, \mu_\ell^{\otimes n})$, $\lambda(\alpha)
  \leq N$.  Let $N'$ be the maximum of $N_n$ for $ 1 \leq n < M$.
 Thus  corollary follows from (\ref{theorem:krashen1}). 
\end{proof}

Let $R$ be a discrete valuation ring with field of fractions $K$ and
residue field $\kappa$. Let $\ell$ be a prime not equal to
char$(\kappa)$.  Then there is a residue homomorphism 
$\partial : H^n(K, \mu_\ell^{\otimes m} ) \to H^{n-1}(\kappa,
\mu_\ell^{\otimes (m-1)})$ with kernel $H^n_{\et}(R,\mu_\ell^{\otimes
  m})$. 

Let $A$ be an integral domain with field of fractions $K$.
Let $\ell$ be a prime which is a unit in $A$.  We have the natural map
$H^n_\et(A, \mu_\ell^{\otimes m}) \buildrel{\iota}\over{\to} H^n(K,
\mu_\ell^{\otimes m})$. An 
element $\alpha$ of $H^n(K, \mu_\ell^{\otimes m})$ is said to be {\it
  unramified} on $A$ if $\alpha$ is in the image of $\iota$. 
 Suppose that $A$ is a regular ring. 
For each height one prime ideal $P$ of $A$, we have the residue
homomorphism  $\partial_P : H^n(K, \mu_\ell^{\otimes m}) \to
H^{n-1}(\kappa(P), \mu_\ell^{\otimes (m-1)})$, where $\kappa(P)$ is
the residue field at $P$.  We have the following

\begin{theorem} (\cite[7.4]{auslander-goldman}) 
\label{auslander-goldman}
Let $A$ be a  regular two dimensional integral domain
  with field of fractions $K$ and $\ell$ a prime which is a unit in
  $A$.  The sequence 
$$0 \to H^2_\et(A, \mu_\ell) \to H^2(K, \mu_\ell) \to \oplus_{P \in
  Spec(A)^{(1)}} H^1(\kappa(P), \Z / \ell\Z)
$$
is exact, where Spec$(A)^{(1)}$ is the set of height one prime ideals of
$A$. 
\end{theorem}

We now recall a few notation from (\cite{hhk:patching_qf_csa}). 
Let $R$ be  a complete discrete valuation ring with field of fractions
$K$ and   residue field $\kappa$.  Let $F$ be the  function field of a
curve over $K$. Let $\XX$ be a regular proper model of $F$ over $R$
and $X$ its reduced special fiber.   For any codimension one point $\eta$ of
$\XX$, let $F_\eta$ be the completion of $F$ at the discrete valuation
of $F$ given by $\eta$ and $\kappa(\eta)$ the residue field at $\eta$.
For a closed point $P$ of $\XX$, let $F_P$ be the field of fractions
of the completion of the local ring at $P$ and $\kappa(P)$ the residue
field at $P$.  Let $U$ be an open subset of $X$. Let $R_U$ be the
ring of all those functions in $F$ which are regular on $U$. Then $R
\subset R_U$.  Let $t$ be a parameter in $R$. Let $\hat{R}_U$ be the
completion of $R_U$ at the ideal $(t)$. Let $F_U$ be the field of
fractions of $\hat{R}_U$. 

Let $A$ be a regular  integral domain with field of fractions $F$. For  a
maximal ideal $m$ of $A$, let $\hat{A}_m$ denote the completion of the
local ring $A_m$ and $F_m$ the field of fractions of $\hat{A}_m$.

\section{Uniform bound - bad characteristic case  } 

Let $K$ be a complete discretely  valued field with residue
  field $\kappa$.  Let $p = char(\kappa)$.   In this section we show that  there is a   
$(2, p)$-uniform bound  for $K(t)$ which depends only on   $[\kappa : \kappa^p]$. 
 
First we recall the following two results from (\cite{ps:period_index}).

\begin{theorem}
\label{theorem:ps1}
 \cite[2.4]{ps:period_index}  Let $R$ be a complete discrete valuation 
  ring  with field of fractions $K$ and 
  residue field $\kappa$. Suppose that char$(K) = 0$, char$(\kappa)
  = p >0$ and $[\kappa : \kappa^p] = p^d$. Let $\pi \in R$ be a
  parameter and $u_1,\cdots ,u_d \in R^*$ units such that $\kappa =
  \kappa^p(\overline{u}_1, \cdots , \overline{u}_d)$, where for any
$u \in R$, $\overline{u}$ denotes the image of $u$ in $\kappa$. 
 Suppose that $K$ contains a primitive
  $p^{\rm th}$ root of unity.  Then any $\alpha \in H^2(K, \mu_p) $ splits over
  $K(\sqrt[p]{\pi}, \sqrt[p^2]{u_1}, \cdots ,  \sqrt[p^2]{u_{d-1}}, \sqrt[p]{u_d})$. In
  particular if $d > 0$, $p^{2d}$ is a $(2, p)$-uniform bound for $K$
  and if $d = 0$,  $p$ is a $(2, p)$-uniform bound for $K$. 
\end{theorem}

\begin{prop} 
\label{prop:ps2} 
(\cite[3.5]{ps:period_index}) Let $A$ be a complete regular local ring of dimension 2 with field
  of fractions $F$ and residue field $\kappa$. Suppose that char$(F) =
  0$, char$(\kappa) = p > 0$ and $[\kappa : \kappa^p] = p^d$. 
Let $\pi, \delta \in A$ and $u_1, \cdots , u_d \in A^*$ such that the
maximal ideal $m$ of $A$ is generated by $\pi$ and $\delta$, and 
$\kappa = \kappa^p(\overline{u}_1, \cdots ,
\overline{u}_d)$, where for any $u \in A$, $\overline{u}$ denotes the
image of $u$ in $\kappa$. 
 Suppose that $F$ contains a primitive
  $p^{\rm th}$ root of unity. Then any $\alpha \in H^2(F, \mu_p)$
  which is unramified on $A$ except possibly at $(\pi)$ and $(\delta)$
  splits over $F(\sqrt[p]{\pi},  \sqrt[p]{\delta}, \sqrt[p^2]{u_1},
  \cdots ,   \sqrt[p^2]{u_d})$.   
\end{prop}

\begin{lemma} 
\label{lemma:choice_of_f_g}
Let $\XX$ be a regular integral  two dimensional scheme
  and $F$ its function field.  
Suppose $C$ and $E$ are regular curves on $\XX$ with normal
 crossings.  Let $\PP$ be a finite set of closed points of
 $\XX$ with $\PP \subset C \cup E$.  Then there exist
  $f, g \in F^*$ such that the maximal ideal at $P$ is  generated
  by $f$ and $g$ for each $P \in \PP$ and $f$ (resp.  $g$) defines $C$
  (resp. $E$)   at each $P \in \PP \cap C$ (resp. $P \in \PP \cap E$). 
\end{lemma} 

\begin{proof}  

Let $R$ be the semi-local ring at all $P \in
  \PP$. Since   $R$ is a unique factorisation domain, there exist  
$f_1, g_1 \in R$  such that div$_{Spec(R)}(f_1) = C\mid _ {Spec(R)}$
and div$_{Spec(R)}(g_1) = E\mid _ {Spec(R)}$.
For  each $P \in \PP$, let $m_P$ be the maximal ideal of the local ring $R_P$ at $P$.
Then for each $P \in C\cap E \cap \PP$, $m_P = (f_1, g_1)$. 
Let $P \in \PP$. Suppose $P \not\in C$. Then $P \in E$.
 Since $E$ is regular on $\XX$,  by the choice of $g_1$, there exists
 $\theta_P \in m_P$ such that $m_P = (\theta_P, g_1)$.  By the Chinese
 remainder theorem, there exists $\pi_P \in m_P$ such that $\pi_P
 \not\in m_Q$ for all $Q \in \PP$, $Q \neq P$ and $\pi_p = \theta_P$
 module $m_P^2$. Then $m_P = (\pi_P, g_1)$.  
 Similarly for each $P \not\in E$, choose $\delta_P 
 \in m_P$ such that $m_P = (f_1, \delta_P)$ and $\delta_P \not\in m_Q$
 for all $Q \in \PP$, $Q \neq P$.   Let
$$f_2  =   \prod_{P \in \PP \setminus C} \pi_P, ~~f = f_1f_2$$
and 
$$g_2 =  \prod_{P \in \PP \setminus E} \delta_P, ~~ g = g_1g_2.$$
Then $f_2$ and $g_2$ are units at all $P \in C\cap E$. 
We claim that $f$ and $g$ have the required properties.
Let $P \in \PP$. Suppose $P \in C \cap E$. Then by the choices 
$f_2$ and $g_2$, they  are units at $P$ and $m_P = (f_1, g_1)$. In particular
$m_P  = (f, g)$ and $f, g$ define $C$ and $E$ respectively at $P$.  
Suppose that $P \not\in C$. Then $f_1$ and $g_2$ are
units at $P$ and  $f_2 = \pi_P u_P$ for some unit  $u_P$ at
$P$. Since $m_P = (\pi_P, g_1)$, we have $m_P = (f, g)$. Since $g_1$
defines $E$ at $P$ and $g_2$ is a unit at $P$, $g$ defines $E$ at $P$.   
Similarly if $P \not\in E$, then $m_P = (f, g)$ and $f$ defines $C$ at
$P$.
\end{proof}

\begin{theorem} 
\label{theorem:bad_char}
Let $K$ be a complete discretely  valued field  with residue field $\kappa$.
  Suppose that char$(K) = 0$
and char$(\kappa) = p > 0$ and $[\kappa : \kappa^p] = p^d$. Assume that
$K$ contains  a primitive $p^{th}$ root of unity.  Then $p^{4d+4}$ is a
$(2, p)$-uniform bound  for $K(t)$.  
\end{theorem}

\begin{proof}  Let $F$ be a finite extension of $K(t)$. 
 Let $\alpha_1, \cdots , \alpha_m \in H^2(F,\mu_p)$.
Let ${\XX}$ be a regular proper model of $F$ over the ring of integers
$R$ of $K$ such that 
the support of  ram$(\alpha_i)$ for all $i$ and the special fiber is
contained in $C \cup E$, where $C$ and $E$ are 
regular curves  on $\XX$   having only  normal crossings. 
 Let $\eta$ be  the generic point  of an
irreducible component  $X_\eta$ of the special fiber of $\XX$. Then
$\kappa(\eta)$ is a function field in one variable over $\kappa$.   
Since $[\kappa : \kappa^p] = p^d$,  $[\kappa(\eta) : \kappa(\eta)^p] =
d+1$ (\cite[A.V.135, Corollary
3]{bourbaki:algebra ll}).  Let $\pi_\eta$ be a parameter at $\eta$ and
$u_{\eta, 1},  \cdots ,  u_{\eta, d+1} \in F^*$ be  lifts of a
$p$-basis of $\kappa(\eta)$. Then, 
by (\ref{theorem:ps1}),   $\alpha \otimes F_\eta(\sqrt[p^2]{u_{\eta, 1}}, 
\sqrt[p^2]{u_{\eta, d}}, \sqrt[p]{u_{\eta, d+1}}, \sqrt[p]{\pi_\eta}) = 0$ for all $i$.   
Let $f \in F^*$ be chosen such that
$\nu_\eta(f) = 1$ for all $\eta$. By the Chinese remainder theorem,
choose  $u_1,  \cdots , u_d \in F^*$ units at each $\eta$ such that  $\overline{u}_j = 
\overline{u}_{\eta, j} \in
\kappa(\eta)$.   Then $\alpha_i \otimes F_\eta(\sqrt[p]{f}, 
 \sqrt[p^2]{u_1}, \cdots , \sqrt[p^2]{u_{d}}, \sqrt[p]{u_{d+1}}) = 0$ for all $i$.
  By (\cite[5.8]{hhk:lgp_torsors}, \cite[1.17]{kmrt}), there exists a non-empty open set $U_\eta$
of the component $X_\eta$ of the special fiber, such that $\alpha \otimes F_{U_\eta}(\sqrt[p]{f}, 
 \sqrt[p^2]{u_1}, \cdots , \sqrt[p^2]{u_{d}}, \sqrt[p]{u_{d+1}}) 
 = 0$ for all $i$. 

Let ${\PP}$ be the finite set of closed points of ${\XX}$ which are not in  
$U_\eta$ for any $\eta$.  Let $A$ be the semi-local ring at the points
of $\PP$. For $P \in \PP$, let $A_P$ be the local ring at $P$. 
Since the ramifications of $\alpha_i$ for all $i$ are in 
normal crossings, for each $P  \in {\PP}$,  the maximal ideal $m_p$ at
$P$ is $(\pi_P, \delta_P)$ for some $\pi_P$ and $\delta_P$ such that
$\alpha_i$ is unramified on $A_P$ except possibly at $(\pi_P)$ and
$(\delta_P)$.  Since the residue field $\kappa(P)$ at  $P$ is a finite
extension of $\kappa$ and $[\kappa : \kappa^p] = p^d$, $[\kappa(P) :
\kappa(P)^p] = p^d$ (\cite[A.V.135, Corollary
3]{bourbaki:algebra ll}). Let $v_{P, 1}, \cdots , v_{P, d} \in A_P^*$
be lifts of a $p$-basis of $\kappa(P)$. 
By the Chinese remainder theorem, choose $h_1, \cdots , h_d \in A^*$
such that $h_i = v_{P, i}$ modulo the maximal ideal at $P$ for all $P
\in \PP$ and $1\leq i
\leq d$. 
By (\ref{lemma:choice_of_f_g}), there exist $g_1, g_2 \in F^*$ such
that   for any point $P \in \PP$, we have  $m_P = 
(g_1, g_2)$ and $g_1$ defines $C$ at all $P \in \PP \cap C$ and $g_2$
defines $E$ at all $P \in \PP \cap E$. In particular,  each $\alpha_i$
is unramified on $A_P$ except possibly at $(g_1)$ and $(g_2)$. 
Then, by (\ref{prop:ps2}),  $\alpha_i \otimes F_P(\sqrt[p^2]{h_1},
\cdots , \sqrt[p^2]{h_d}, \sqrt[p]{g_1},   \sqrt[p]{g_2}) = 0$ for all $i$.

Let $L = F(\sqrt[p]{f}, 
 \sqrt[p^2]{u_1}, \cdots , \sqrt[p^2]{u_{d}}, \sqrt[p]{u_{d+1}}, \sqrt[p^2]{h_1}, \cdots , 
 \sqrt[p^2]{h_d}, \sqrt[p]{g_1},  \sqrt[p]{g_2})$.
We claim that $\alpha_i \otimes L = 0$ for all $i$.
Let   ${\YY}$ be a regular proper model of $L$ and $Y$ its special fiber. 
 Let $y $ be  a point of  $Y$.  If $y$ lies over a  generic point of a
 component of the special    fiber $X$ of ${\XX}$, then by the choice
 of  $f, u_1, \cdots , u_d$, $\alpha_i  \otimes L_y = 0$.
 Suppose the image of $y$ lies over a closed point  $P$ of  $X$. 
 Suppose  $P \in U_\eta$ for some $\eta$. Then $F_{U_\eta} \subset F_P$
 and hence once again by the choice of  $f, u_1, \cdots , u_d$,
 $\alpha_i \otimes L_y = 0$ for all $i$. 
 Suppose $P \not\in U_\eta$ for all $\eta$. Then $P \in {\PP}$.
 Since $F_P \subset L_y$, by the choices of $h_1, \cdots , h_d, g_1, g_2$,
 $\alpha_i \otimes L_y = 0$ for all $i$.
 
 Since $[L : F] \leq p^{4d + 4}$, the theorem follows.
\end{proof}

\begin{cor} 
\label{cor:bad_char_2}
Let $K$ be a complete discretely  valued field  with residue field $\kappa$.
Suppose that char$(K) = 0$
and char$(\kappa) = p > 0$ and $[\kappa : \kappa^p] = p^d$. Let
$\zeta$ be a primitive   $p^{th}$ root of unity.  Then $[K(\zeta) : K]p^{4d+4}$ is a
$(2, p)$-uniform bound  for $K(t)$. 

\end{cor}

\begin{proof} 
\label{cor:bad_char_n}
Let $K' = K(\zeta)$. Then $K'$ is a complete discretely
  valued field with residue field $\kappa$. Let $F$ be a finite
  extension of $K(t)$. 
Let $\alpha_1, \cdots , \alpha_m \in H^2(F, \mu_p^{\otimes 2})$. 
Since  $F'= F(\zeta)$ is also a function field  over  $K'$, by (\ref{theorem:bad_char}), there
  exists an extension $L$ of $F'$ of degree at most $p^{4d+4}$ such
  that $\alpha_i \otimes L  = 0$ for all $i$. Since $[L : F] = [L :
  F'][F' : F]$, the corollary follows. 
\end{proof}
 
The above corollary and (\ref{lemma:uniform_bound_d_d+1}) give the
following

\begin{cor} Let $K$ be a complete discretely  valued field  with residue field $\kappa$.
Suppose that char$(K) = 0$
and char$(\kappa) = p > 0$ and $[\kappa : \kappa^p] = p^d$. Let
$\zeta$ be a primitive   $p^{th}$ root of unity.  Then $[K(\zeta) : K]p^{4d+4}$ is an
$(n, p)$-uniform bound  for $K(t)$ for all $n \geq 2$. 
\end{cor}

\section{Uniform bound - good characteristic case  } 

Let $F$ be the  function field of a $p$-adic curve. In (\cite{saltman:jrms}), Saltman
proved that if  $\ell$ is  a prime not equal
to $p$ and $\alpha_1, \cdots , \alpha_m \in H^2(F, \mu_\ell)$, then there exists
an extension $L$ of $F$ such that $[L : F] \leq \ell^2$ and
$\alpha_{i_L} = 0 $ for all $i$, i.e., $F$  is  $(2,
\ell)$-uniformly bounded.

Let $K$ be a complete discretely  valued field with residue
  field $\kappa$.     
Let $\ell$ be a prime not equal to char$(\kappa)$.
In this section we show that $K(t)$   is  $(2,\ell)$-uniformly bounded
under some conditions on $\kappa$.

\begin{theorem}
\label{theorem:good_char_dvr} Let $K$ be a  complete discretely  valued field with residue
  field $\kappa$. Let $\ell$ be a prime not equal to char$(\kappa)$
  and $n \geq 1$.
If $N$ is an $(n, \ell)$-uniform bound for  $\kappa$, then
$\ell N$ is an  $(n, \ell)$-uniform bound for $K$.
\end{theorem}
 
\begin{proof} Let $L$ be a finite extension of $K$. Then $L$ is a
  complete   discretely valued  field with residue $\kappa'$ a finite
  extension of $\kappa$. Let $R$ be the valuation ring of $R$ and  $\pi \in R$ be
  a parameter.   
Let $\alpha_1, \cdots, \alpha_m \in H^n(L,
  \mu_\ell^{\otimes n})$. Let $S$ be the integral closure of $R$ in
  $L(\sqrt[\ell]{\pi})$.    Then  $S$ is also a complete discrete
  valuation ring with residue field $\kappa'$.  Since  $S/R$ is
  ramified,  $\alpha_i \otimes
  L(\sqrt[\ell]{\pi})$ is unramified at $S$  for each $i$. 
Hence there  exists $\beta_i \in H^n_{\et}(S, \mu_\ell^{\otimes n})$ such that $\beta_i
  \otimes_SK(\sqrt[\ell]{\pi}) = \alpha_i$. Since $N$ is an $(n,
  \ell)$-uniform bound for $\kappa$, there exists an extension $L_0$
  of $\kappa$ of degree at most $N$ such that $\beta_i \otimes_\kappa L_0 =
  0$ for all $i$.  Let $L$ be the extension of $K$ of degree equal to
  $[L_0 : \kappa]$ with  residue field $L_0$.  Let $T$ be the
  integral closure of $R$ in $L(\sqrt[\ell]{\pi})$. Then $T$ is a
  complete discrete valued ring  with residue field $L_0$ and
  $S \subset T$. Since $\beta_i \otimes_\kappa L_0 = 0$, $\beta_i
  \otimes_S T = 0$ for all $i$. In particular $\alpha_i \otimes_K
  L(\sqrt[\ell]{\pi}) = 0$ for all $i$. Since the degree of $L$ over
  $K$ is  equal to the degree of $L_0$ over $\kappa$ and $[L_0 :
  \kappa] \leq N$, we have $[L(\sqrt[\ell]{\pi}) : K] \leq \ell N$.
\end{proof}

\begin{lemma} 
\label{lemma:good_char_dim2}
Let $A$ be a  regular local ring with residue
field $\kappa$ and maximal ideal $m = (\pi,\delta)$. Let $F$ be the field of
fractions of $A$, $\ell$ a prime not equal to char$(\kappa)$. 
Let $B$ (resp. $B'$) be the integral closure of $A$ in
  $F(\sqrt[\ell]{\pi}, \sqrt[\ell]{\delta})$ (resp. $F(\sqrt[\ell]{\pi})$).
 Let $\alpha \in H^2(F, \mu_\ell^{\otimes 2})$. If
$\alpha$ is unramified on $A$ except possibly at $(\pi)$ and
$(\delta)$ (resp. except possibly at $(\pi)$), then $\alpha \otimes_F F(\sqrt[\ell]{\pi},
\sqrt[\ell]{\delta})$ is unramified on $B$ (resp. $\alpha \otimes_F
F(\sqrt[\ell]{\pi})$ is unramified on $B'$).  
\end{lemma}

\begin{proof} By (\cite[3.3]{ps:period_index}),  $B$ is a    regular
  local ring of dimension 2 with residue field $\kappa$. 
Let $P$ be a height one prime ideal of $B$ and $Q = P \cap A$. Then
$Q$ is a height one  prime ideal of $A$. If $Q \neq (\pi)$ and
$(\delta)$, then $\alpha$ is unramified at $Q$ and hence 
$\alpha \otimes F(\sqrt[\ell]{\pi}, \sqrt[\ell]{\delta})$ is
unramified at $P$. Suppose $Q = (\pi)$ or $(\delta)$.  Then $Q$ is ramified in $B$
and hence $\alpha \otimes F(\sqrt[\ell]{\pi}, \sqrt[\ell]{\delta})$ is
unramified at $P$. Since $B$ is a regular local ring of dimension 2
and $\alpha \otimes F(\sqrt[\ell]{\pi}, \sqrt[\ell]{\delta})$ is
unramified at every height one prime ideal of $B$, 
$\alpha \otimes F(\sqrt[\ell]{\pi}, \sqrt[\ell]{\delta})$ is
unramified on $B$ (cf. \ref{auslander-goldman}). 
The other case follows similarly. \end{proof}

\begin{theorem}
\label{theorem:good_char_killing_ramification} 
Let $R$ be a complete discrete valuation ring with
  field of fractions $K$ and    residue field $\kappa$. Let $F$ be the
  function field of a curve over $K$.  Let $\ell$ be a prime not equal
  to char$(\kappa)$ and $\alpha_1, \cdots ,
  \alpha_r \in H^n(K, \mu_\ell^{\otimes m})$. Then there exist $f, g,  h \in F^*$
  such that  each $\alpha_i \otimes F(\sqrt[\ell]{f}, \sqrt[\ell]{g},
  \sqrt[\ell]{h})$ is unramified at all codimension one points of any
  regular proper model  of $F(\sqrt[\ell]{f}, \sqrt[\ell]{g},
  \sqrt[\ell]{h})$. 
\end{theorem}

\begin{proof} Let $\XX$ be a regular proper model of $F$ over $R$  such that 
  the union of  the support of ramification locus of $\alpha_i$, for $1
  \leq i \leq r$,  is  contained in the union of regular curves $C$ and
  $E$ with $C \cup E$ having only normal crossings.  Let $f \in F^*$
  be such that
$${\rm div}_\XX(f) = C + E + F$$
for some divisor $F$ on $\XX$ which does not contain any irreducible
component of $C \cup E$ and  does not pass through any point of $C
\cap E$.  Let $g \in F^*$ be such that 
$${\rm div}_\XX(g) = C  + G$$
for some divisor $G$ on $\XX$ which does not contain any irreducible
component of $C \cup E \cup F$ and  does not pass through any point of $C
\cap E$, $C \cap F$ and $E \cap F$.  Let $h \in F^*$ be such that 
$${\rm div}_\XX(h) =  E + H$$
for some divisor $H$ on $\XX$ which does not contain any irreducible
component of $C \cup E \cup F \cup G$ and  does not pass through any point of $C
\cap E$, $C\cap F$, $C \cap G$, $E \cap F$, $E \cap G$ and $F \cap
G$. 

Let $L = F(\sqrt[\ell]{f}, \sqrt[\ell]{g}, \sqrt[\ell]{h})$ and $\YY$
be  a regular proper model of $L$ over $R$. We claim that $\alpha_i
\otimes_F L$ is unramified on $\YY$. Let $y \in \YY$ be a codimension
one point of $\YY$. Then $y$ lies over a point $x$ of $\XX$. 
If $x$ is not on $C$ or
$E$, then each $\alpha_i$ is unramified at $x$ and hence $\alpha
\otimes_F L$ is unramified at $y$. Assume that $x \in C\cup E$.

Suppose that $x$ is a codimension one point of $\XX$.  Since  $x$ is on $C$ or
$E$, by the choice of  $f$, $f$  is a parameter at $x$ and hence $L/K$ is unramified at
$x$.  In particular $\alpha_i \otimes_F L $ is unramified at $y$.

Suppose that $x$ is a closed point of $\XX$. Suppose $x \in C$ and $x
\not\in E$.  Let $A_x$ be the local ring of $\XX$ at $x$ and $S_y$ be
the local ring of $\YY$ at $y$.   
Suppose $x \not\in F$. Then the maximal ideal $m_x$ at 
$x$ is $(f, \delta_x)$ for some $\delta_x \in m_x$ and each $\alpha_i$
is unramified on the local ring at $x$ except possibly at $(f)$.  Thus,
by (\ref{lemma:good_char_dim2}),  each  $\alpha_i \otimes_F
F(\sqrt[\ell]{f})$ is unramified at the integral closure of $A$ in
$F[\sqrt[\ell]{f}]$. Since the integral closure of $A$ in
  $F[\sqrt[\ell]{f}]$ is contained in $S_y$,  each $\alpha_i\otimes_F \otimes
  L$ is unramified at $S_y$. Suppose $x \in F$. If $x \not\in G$, the
  as above each $\alpha_i \otimes_F L$ is unramified at $S_y$. If $x
  \in F\cap G$, then by the choice $h$, $x \not\in H$ and hence as
  above,  each $\alpha_i \otimes_F L$ is unramified at $S_y$.
  Similarly if $x \in E$ and $x \not\in C$, then each $\alpha_i
  \otimes_F L$ is unramified at $S_y$.

Suppose $x \in C \cap E$.   Then $x \not\in G$ and $x \not\in H$. In
particular, $m_x = (g, h)$ and  each $\alpha_i$ is unramified on $A_x$
except possibly at $(g)$ and $(h)$. As above each $\alpha_i \otimes_F
L$ is unramified at $S_y$. 
\end{proof}

\begin{lemma} 
\label{lemma:crt}
Let $A$ be a  semi-local regular domain   with field of fractions
  $F$.    For each maximal ideal $m$
  of $A$, let $\mathfrak{s}(m)$ be a separable  finite extension of  the residue field
  $\kappa(m)$ at $m$ of degree $N_m$.  Let $N$ be a  common
  multiple of $N_m$,  $m$ varying  over all maximal ideals  of $A$. Then there exists  
  an extension  $E$ of $F$ of degree at most $N$ such that
 for each  maximal ideal $m$ of $A$ and for each maximal ideal $m'$ of
 the integral   closure $B_m$ of  $A_m$ in $E$,  $B_m/m'$  contains a
 field   isomorphic to $\mathfrak{s}(m)$.   
\end{lemma} 

\begin{proof}    Since $\mathfrak{s}(m)$ is a finite separable extension
  of  $\kappa(m)$, there exists $\theta_m
   \in \mathfrak{s}(m) $ such that $\mathfrak{s}(m) = \kappa(m)(\theta_m)$.
 Let $f_m(X) \in \kappa(m)[X]$ be the minimal polynomial of
   $\theta_m$ over $\kappa(m)$. Then the degree of $f_m(X)$ is $N_m$.  
 Let $f(X) \in A[X]$ be
  a monic polynomial of degree $N$ such that $f(X) = f_m(X)^{N/deg(f_m)}$ modulo
  $m$ for each maximal ideal $m$ of $A$.   Let $g(X)$ be any monic
  irreducible factor of $f(X)$ over $A$. Let $E = F[X]/(g(X))$. We claim that
  $E$ has the required property. 

Let $m$ be a maximal ideal of $A$.  By the choice of $f(X)$ and
$g(X)$, we have $g(X) = f_m(X)^{r_m}$ modulo $m$ for some $r_m \geq
1$.  Let $B_m$ be the integral closure of $A_m$ in $E$. Since $g(X)$
is monic, $A_m[X]/(g(X))$ is isomorphic to a subring of $B_m$.  Let
$g_m(X) \in A_m[X]$ be a monic polynomial with $g_m(X) = f_m(X)$
modulo $m$. Since $f(X) = f_m(X)^{r_m} = g_m(X)^{r_m}$ modulo $m$,
the ideal $\tilde{m}$ of $ A_m[X]/(g(X))$ generated by $m$ and $g_m(X)$ is a maximal 
ideal  with $(A_m[X]/(g(X)))/\tilde{m} \simeq \mathfrak{s}(m)$. 
Since $B_m$ is integral over a subring isomorphic to
$A_m[X]/(g(X))$, for every maximal ideal $m'$ of $B_m$, $B_m/m'$
contains a subfield isomorphic to $(A_m[X]/(g(X)))/\tilde{m} \simeq \mathfrak{s}(m)$.
\end{proof} 

\begin{cor} 
\label{cor:crt} 
Let $A$ be a  semi-local regular domain   with field of fractions
  $F$.  Let $\ell$ be a prime. Suppose that $\ell$ is a unit in $A$.
  Let $\beta \in H^2_\et(A, \mu_\ell)$. Suppose that for every maximal ideal 
  $m$ of $A$, there  exists a finite  separable extension
  $\mathfrak{s}(m)$ of $\kappa(m)$ of  
  degree $N_m$ such that $\beta \otimes_A \mathfrak{s}(m) = 0$. 
  Let $N$ be a common multiple of $N_m$,  where $m$ varies over
  maximal ideals of $A$.  Let  $E$ be 
  the field constructed in  (\ref{lemma:crt}). Then for any maximal 
  ideal $m$ of $A$,  $\beta \otimes_A (E \otimes_F F_m) = 0$. 
 \end{cor}

\begin{proof}   Let $B$ be the integral closure of $A$ in $E$. 
Let  $m$ be a maximal  ideal   of $A$. 
and  $\hat{A}_m$  the completion of $A$ at $m$. 
Then,  $B \otimes_A \hat{A}_m$ is complete and by the choice of $E$, 
$B \otimes_A \hat{A}_m$ modulo its radical is   isomorphic to a
product of fields with  
each factor containing a field  isomorphic to $\mathfrak{s}(m)$.
Since $\beta \otimes \mathfrak{s}(m) = 0$, it follows that $\beta
\otimes B \otimes  \hat{A}_m = 0$. 
Since $E$ is the field of fractions of $B$,  $\beta \otimes_A (E
\otimes_F F_m) = 0$. 
 \end{proof}

\begin{theorem} 
\label{theorem:good_char}
Let $K$ be complete discretely  valued field with residue field
$\kappa$. 
Let $\ell$ be a prime not equal to char$(\kappa)$. 
Suppose that     $N_1$ is
a $(2, \ell)$-uniform bound for   $\kappa(t)$
and $N_2$ is a $(2, \ell)$-uniform bound for   $\kappa$.     
Then  $\ell^3 (N_1!)(N_2!)$ is 
a $(2, \ell)$-uniform  bound for $K(t)$. 
\end{theorem}

\begin{proof}  Let $F$ be a finite extension of $K(t)$.  
Let $\alpha_1, \cdots , \alpha_m \in H^2(F, \mu_\ell^{\otimes 2})$. 
 Then, by  (\ref{theorem:good_char_killing_ramification}),  
there exist $f, g, h \in F^*$ such that  $\alpha_i \otimes 
F(\sqrt[\ell]{f}, \sqrt[\ell]{g}, \sqrt[\ell]{h})$ is 
unramified at every codimension one point of any regular proper 
model of $F(\sqrt[\ell]{f}, \sqrt[\ell]{g}, \sqrt[\ell]{h})$ over $R$. 
Let $L  = F(\sqrt[\ell]{f}, \sqrt[\ell]{g}, \sqrt[\ell]{h})$. Let  ${\YY}$ be a 
regular proper model of $L$ over $R$ and $Y$ the reduced special fiber of $\YY$.

Let $\eta$ be   the generic point   of an irreducible component $Y_\eta$ of 
  $Y $.  Let $A_\eta$ be the local ring at $\eta$.  
 Since each  $\alpha_i \otimes_F L$ is unramified at $\eta$, 
there exists $\beta_i \in H^2_\et(A_\eta, \mu_\ell^{\otimes 2})$ such that $\beta_i 
\otimes_{A_\eta} F_\eta = \alpha_i$ for $1 \leq i \leq m$.  Since $\kappa(\eta)$ is a finite 
extension of  $\kappa(t)$ and $N_1$ is a $(2, \ell)$-uniform bund for
$\kappa(t)$, there exists a  
finite  extension $\mathfrak{s}(\eta)$ of degree at most $N_1$ such that 
$\beta_i \otimes_{A_\eta} \mathfrak{s}(\eta) = 0$ for all $i$.     
By (\ref{lemma:sep_ext}), we assume that $\mathfrak{s}(\eta)$ is
separable over $\kappa(\eta)$.  
Let $A$ be the semi-local ring at the generic points of all
irreducible components of  
the special fiber $Y$ of $\YY$.  Then $A$ is a semi-local regular ring with field
of fractions $L$. By (\ref{lemma:crt}), there exists a field extension
$L_1$ of $L$ of degree  
at most $N_1!$ such that for every maximal ideal $m'$ of the integral
closure $B_m$ of $A_m$ 
in $L$, $B_m/m'$ contains a subfield  isomorphic to
$\mathfrak{s}(\eta)$. Hence, by (\ref{cor:crt}),   
$\alpha_i \otimes (L_1 \otimes_{L} L_{\eta}) = \beta_i  \otimes (L_1
\otimes_{L} L_{\eta}) = 0$  
for all $i$.
  By (\cite[5.8]{hhk:lgp_torsors},
\cite[1.17]{kmrt}), 
there exists a non-empty open set $U_\eta$
of the component $Y_\eta$ of the special fiber $Y$,  such that $\alpha_i
\otimes L_1  \otimes L_{{U_\eta}}   = 0$ for all $i$. 

Let $\PP$ be the finite set of closed points of ${\YY}$ which are not in  
$U_\eta$ for any $\eta$.   Let $A_\PP$ be the regular semi-local ring at the closed 
points of $\PP$.  Since each $\alpha_i $ is unramified on $\YY$, there 
exists $\beta_i \in H^2_\et(A_\PP, \mu_\ell^{\otimes 2})$ such that 
$\beta_i \otimes L = \alpha_i \otimes L$. 
Let $P \in \PP$.   Since the residue field $\kappa(P)$ at $P$ is a
finite extension of $\kappa$, by the assumption on $\kappa$, 
there exists an extension $\mathfrak{s}(P)$ of $\kappa(P)$ of degree at most 
$N_2$ such that $\beta_i \otimes \mathfrak{s}(P) = 0$ for all $i$. Once again, by 
(\ref{lemma:sep_ext}), we assume that  each $\mathfrak{s}(P)$ is a separable extension 
of $\kappa(P)$.  Let $L_2$ be as in (\ref{lemma:crt}).
Then, as above, by (\ref{cor:crt}), $\alpha_i \otimes (L_2 \otimes L_{P}) = 0$ for all $i$. 
 
Let $L = LL_1L_2$. 
Then as in (\ref{theorem:bad_char}), $L \otimes \alpha_i = 0$ for all $i$. 
Since $[L : F] \leq \ell^3(N_1!)(N_2!)$, the theorem follows.
\end{proof}

\begin{cor} 
\label{cor:good_char_2}
Let $K$, $\kappa$, $\ell$, $N_1$ and $N_2$ be as in
  (\ref{theorem:good_char}). 
Let $\zeta$ be a primitive  $\ell^{\rm th}$ root
  of unity.  Then $[K(\zeta), K]\ell^3(N_1!)(N_2!)$ is a $(2, \ell)$-uniform
  bound of $K(t)$.
\end{cor}

The above corollary and (\ref{lemma:uniform_bound_d_d+1}) gives the
following

\begin{cor} 
\label{cor:good_char_n}
Let $K$, $\kappa$, $\ell$, $N_1$ and $N_2$ be as in
 (\ref{theorem:good_char}). 
Let $\zeta$ be a primitive  $\ell^{\rm th}$ root
  of unity.  Then $[K(\zeta), K]\ell^3(N_1!)(N_2!)$ is an $(n, \ell)$-uniform
  bound  of $K(t)$ for all $n \geq 2$.
\end{cor}

\section{  Symbol length and $u$-invariant}
 
 \begin{theorem}
\label{theorem:symbol_length1}
 Let $K$ be a complete discretely  valued field with
  residue field $\kappa$. Let $\ell$ be a prime not equal to
  char$(\kappa)$.  Suppose that  there exist integers $N_1$ and $N_2$  such
  that $\kappa(t)$ is  $(2, \ell)$-uniformly bounded by $N_1$and
  $\kappa$ is $(2, \ell)$-uniformly bounded by $N_2$.  Let $n \geq
  2$. Then there
  exists an integer $M_n$ which depends only on $N_1$, $N_2$ and $n$
  such that  for every finite extension $F$ of $K(t)$ and 
    for all $\alpha
  \in H^n(F, \mu_\ell^{\otimes n})$,  $\lambda(\alpha) \leq M_n$.  
\end{theorem}

\begin{proof} By (\ref{cor:good_char_2}), $K(t)$ is $(2,
  \ell)$-uniformly bounded by $N = (\ell-1)\ell^3(N_1!)(N_2!)$. 
  Hence any finite extension $F$ of $K(t)$ is also $(2, \ell)$-uniformly bounded by 
  $N$. The theorem follows from (\ref{cor:symbol_length}).
 \end{proof}

\begin{theorem} 
\label{theorem:symbol_length2} 
Let $K$ be a complete discretely  valued field with
  residue field $\kappa$. Let $p = $ char$(\kappa)$. Suppose that
  char$(K) = 0$, $p > 0$ and $[\kappa : \kappa^d] = p^d$.  
Then  there exists an integer $M$, which depends only on
  $d$ such that for any finite extension $F$  of
  $K(t)$ and for all $\alpha \in H^n(F, \mu_p^{\otimes n})$,  $n
  \geq 1$,   $\lambda(\alpha) \leq M$.  
\end{theorem}

\begin{proof}  By (\ref{cor:bad_char_2}), $K(t)$ is $(2,
  p)$-uniformly bounded by $(p-1)p^{4d + 4}$. 
  Let $F$ be a finite extension of $K(t)$. Then $F$ is also $(2, p)$-uniformly 
  bounded by $(p-1)p^{4d+4}$. 
  Let $n \geq 1$. 
By   (\ref{cor:symbol_length}),   there exists an integer $N_n$, which
depends only  on $d$ and $n$ such that for   all $\alpha \in H^n(F,   \mu_p^{\otimes n})$,
  $\lambda(\alpha) \leq N_n$.  Since the $p$-cohomological dimension of
  $K$ is at most $d + 2$ (\cite{g-o}) and $F$ is  a function field in
  one variable over $K$, the $p$-cohomological dimension of $F$ is $d +
  3$. Hence $H^n(F, \mu_p^{\otimes n}) = 0$ for all $n \geq
  d+4$. Let $N$ be the maximun of $N_n$ for $2 \leq n \leq d+4$. Then
  $\lambda(\alpha) \leq  N$ for all $\alpha \in H^n(F,
  \mu_p^{\otimes n})$ and $n \geq 2$.
\end{proof}

 \begin{theorem}\label{theorem:u_invariant} 
Let $K$ be complete discretely  valued field with residue field
$\kappa$ and $F$ a function field of a curve over $K$. 
Suppose that char$(K)  = 0$, char$(\kappa) = 2$ and $[\kappa : \kappa^2]$ is finite.  
Then there exists an integer $M$ which depends only on $[\kappa :
\kappa^2]$ such that  for any finite extension $F$ of $K(t)$, $u(F) \leq M$.  
\end{theorem}

\begin{proof}   The theorem follows from
  (\ref{theorem:symbol_length2}) and  (\ref{cor:uniformbound_uinvariant}).

\end{proof}

We end with the following

\begin{question} Let $L$ be a field of characteristic not equal to 2
  with $u(L)$ finite.  Is  the Brauer group of $L$ uniformly
  2-bounded?

\end{question}

\newpage
\providecommand{\bysame}{\leavevmode\hbox to3em{\hrulefill}\thinspace}

\end{document}